\documentclass{article}

\usepackage{amsmath, amsthm, amssymb}
\usepackage{enumerate}
\usepackage{lineno}
\usepackage{hyperref}
\usepackage{authblk}
\usepackage[ruled,vlined,linesnumbered,resetcount,algosection]{algorithm2e}
\usepackage{color}
\usepackage{tikz}
\usepackage{comment}
\pgfdeclarelayer{edgelayer}
\pgfdeclarelayer{nodelayer}
\pgfsetlayers{edgelayer,nodelayer,main}
\usepackage[a4paper, total={6in, 8in}]{geometry}

\newtheorem{theorem}{Theorem}[section]
\newtheorem{definition}[theorem]{Definition}
\newtheorem{remark}[theorem]{Remark}
\newtheorem{lemma}[theorem]{Lemma}

\DeclareMathOperator{\BoxProduct}{\mathbin{\Box}}
\DeclareMathOperator{\diam}{diam}
\DeclareMathOperator{\rad}{rad}
\DeclareMathOperator{\girth}{g}
\DeclareMathOperator{\e}{ecc}
\DeclareMathOperator{\aug}{aug}

\newcommand{\svSpan}[1]{\sigma^{\boxtimes}_V(#1)}
\newcommand{\seSpan}[1]{\sigma^{\boxtimes}_E(#1)}
\newcommand{\dvSpan}[1]{\sigma^{\times}_V(#1)}
\newcommand{\deSpan}[1]{\sigma^{\times}_E(#1)}
\newcommand{\cvSpan}[1]{\sigma^{\BoxProduct}_V(#1)}
\newcommand{\ceSpan}[1]{\sigma^{\BoxProduct}_E(#1)}

\begin{document}
\title{Graphs with span 1 and shortest optimal walks}
\author{
\and Tanja Dravec$^{a,b}$\thanks{tanja.dravec@um.si}, Mirjana Mikalački$^{c}$\thanks{mirjana.mikalacki@dmi.uns.ac.rs}, Andrej Taranenko$^{a,b}$\thanks{andrej.taranenko@um.si}
}

\date{\today}

\maketitle

\begin{center}
$^a$ Faculty of Natural Sciences and Mathematics, University of Maribor, Slovenia \\
$^b$ Institute of Mathematics, Physics and Mechanics, Ljubljana, Slovenia\\
$^c$  Department of Mathematics and Informatics, Faculty of Sciences, University of Novi Sad, Serbia
\end{center}

\begin{abstract}
A span of a given graph $G$ is the maximum distance that two players can keep at all times while visiting all vertices (edges) of $G$ and moving according to certain rules, that produce different variants of span. We prove that the vertex and edge span of the same variant can differ by at most 1 and present a graph where the difference is exactly 1. For all variants of vertex span we present a lower bound in terms of the girth of the graph. Then we study graphs with the strong vertex span equal to 1. We present some nice properties of such graphs and show that interval graphs are contained in the class of graphs having the strong vertex span equal to 1. Finally, we present an algorithm that returns the minimum number of moves needed such that both players traverse all vertices of the given graph $G$ such that in each move the distance between players equals at least the chosen span of $G$. 
\end{abstract}

\section{Introduction and basic definitions}
The notion of span of a continuum has been introduced by Lelek \cite{Lelek} in 1964 and has been a popular topic of research in continuum theory. The idea of a span of a continuum was recently translated to an equivalent in graph theory by Banič and Taranenko \cite{BaTa23}. It can be interpreted in the language of games on graphs as follows. Two players, say Alice and Bob, are moving through a given graph, restricted by some movement rules, and wish to traverse all vertices (or edges) of the graph. At each point in time, the distance between the two players on the graph is measured. The minimum obtained over all moments in time is the safety distance Alice and Bob were able to maintain. The question asked is, what is the maximum possible safety distance the two players are able to maintain at each point in time. Depending on the set of movement rules and whether they wish to visit all vertices or edges, the obtained value corresponds to a span of the given graph.

The movement rules commonly used in games in graphs are: \emph{the traditional movement rules} where at each point in time the two players independently of each other decide to either stay at their current vertex or move to an adjacent one,  \emph{the active movement rules} where at each point in time each player must move to an adjacent vertex, and finally  \emph{the lazy movement rules} where at each point in time exactly one player must move to an adjacent vertex.

In the seminal paper, Banič and Taranenko \cite{BaTa23} formally introduce the notion of spans of a graph. They present a characterisation of spans with respect to subgraphs of graph products and show that the value of a chosen span can be obtained in polynomial time. Also, 0-span graphs are characterised for each span variant. Erceg et al. \cite{Erceg23} continued the study of spans of a graph. In their paper the relation between different vertex span variants was studied, and spans were also studied for specific families of graphs. Šubašić and Vojković \cite{SuVo24} determined the values of all variants of vertex spans for multilayered graphs and their subclasses, i.e. multilayered cycles and multilayered paths. Lasty, in \cite{SuVoAX} also studied the minimum number of moves players must make in order to traverse all vertices/edges of a given graph while still maintaining the distance at least the corresponding span at all times.

In this paper we continue the study of spans of a graph. In the remainder of this section we define basic notation and state the results needed throughout the paper. In Section \ref{sec:bounds} we show that the vertex and the edge span of the same variant can differ at most by one, and show that such graphs exist. For all vertex span variants we also provide a lower bound in terms of the girth of the graph. In Section \ref{sec:span1} we study graphs with the strong vertex span equal to 1, where we present some nice properties of such graphs and provide a construction of an infinite family of graphs with the strong vertex span equal to 1. In the last section, we present an algorithm for determining the minimum number of steps needed in optimal walks.

Let $G$ be a connected graph and $v$ a vertex of $G$. The \emph{eccentricity} of the vertex $v$, denoted $\e(v)$ is the maximum distance from $v$ to any vertex of $G$. That is, $\e(v)=\max\{d_G(v,u) \mid u \in V(G)\}.$ The \emph{radius} of $G$, denoted $\rad(G)$, is the minimum eccentricity among the vertices of $G$. Therefore, $\rad(G)=\min\{\e(v) \mid v \in V(G)\}.$ The \emph{diameter} of $G$, denoted $\diam(G)$, is the maximum eccentricity among the vertices of $G$, thus, $\diam(G)=\max\{\e(v) \mid v \in V(G)\}.$ 

Let $G$ be a graph and $S\subseteq V(G)$. The subgraph of $G$ induced by $S$ is denoted by $G[S]$. The set $S$ is called a \emph{clique} if $G[S]$ is a complete graph. $S$ is a \emph{cut set} of $G$ if $G-S$ has more components than $G$. A cut set $S$ of a connected graph $G$ is called a \emph{minimal cut set} if no proper subset of $S$ is a cut set of $G$.
For two vertex disjoint graphs $G$ and $H$, the \emph{join} of graphs $G$ and $H$, $G \vee H$, is the graph with the vertex set $V(G) \cup V(H)$ and edge set $E(G) \cup E(H)\cup \{gh \mid g \in V(G),h \in V(H)\}.$

Let $G$ and $H$ be any graphs. A function $f: V(G) \rightarrow V(H)$ is a \emph{weak homomorphism from $G$ to $H$} if  for all $u,v\in V(G)$, $uv\in E(G)$ implies $f(u) f(v) \in E(H)$ or $f(u)=f(v)$. We will use the more common notation $f:G\rightarrow H$ to say that $f: V(G) \rightarrow V(H)$ is a weak homomorphism. A weak homomorphism $f: G \rightarrow H$ is \emph{surjective} if $f(V(G)) = V(H)$. A weak homomorphism $f: G \rightarrow H$ is \emph{edge surjective} if it is surjective and for every $uv \in E(H)$ there exists an edge $xy \in E(G)$ such that $u=f(x)$ and $v=f(y)$. For a weak homomorphism $f: G \rightarrow H$, the \emph{image $f(G)$ of the graph $G$} is the graph defined by $V(f(G)) = \{ f(u)  \mid u \in V(G)\}$ and $E(f(G)) = \{ f(u)f(v)  \mid uv\in E(G) \text{ and } f(u) \neq f(v)\}.$

Let $f,g: G \rightarrow H$ be weak homomorphisms. The \emph{distance from $f$ to $g$} is defined as $m_G(f,g) = \min \{ d_H(f(u), g(u))  \mid u \in V(G) \}$.  If $f,g: G \rightarrow H$ are surjective weak homomorphisms and $G$ is connected, then $m_G(f,g)\leq \rad(H)$ \cite{BaTa23}.

We are now ready to state the definitions of all span variants (depending on the movement rules and the goal: traversing all vertices or edges) as defined in \cite{BaTa23}. 

\begin{definition}\label{def:strongSpans}
Let $H$ be a connected graph. 
Define 
\begin{align*}
  \seSpan{H} = \max \{ m_P(f,g)  \mid &f,g: P \rightarrow H \text{ are edge surjective weak  homomorphisms and $P$ is a path} \}.  
\end{align*}
     
We call $\seSpan{H}$ the \emph{strong edge span} of the graph $H$.

Define 
\begin{align*}
     \svSpan{H} = \max \{ m_P(f,g)  \mid &f,g: P \rightarrow H \text{ are surjective weak homomorphisms and $P$ is a path} \}. 
\end{align*}
 
We call $\svSpan{H}$ the \emph{strong vertex span} of the graph $H$.
\end{definition}

The idea of using weak homomorphisms is easy to explain.  Walks of graphs parameterised by time can be presented by paths, where two adjacent vertices represent consecutive points in time. A mapping from a path to the given graph defined by a weak homomorphism represents one walk on the given graph. Since by a weak homomorphism two adjacent vertices on a path can be mapped to the same vertex, this can be seen as the fact that the player did not move at this point in time; on the other hand, if two adjacent vertices on a path are mapped to two adjacent vertices on the graph, this means the player moved to an adjacent vertex. The condition that we consider only surjective (edge surjective) weak homomorphisms implies that the walk considered is a walk through all vertices (edges) of the graph. Therefore this definition corresponds to the traditional movement rules, in the notation represented by the symbol $\boxtimes$ (see \cite {BaTa23} for a more detailed explanation on reasons for such notations). So two weak (edge) surjective homomorphisms $f$ and $g$ in Definition \ref{def:strongSpans} represent a pair of walks through all (edges) vertices of the graph, one for each player. It is not difficult to see that the value $m_P(f,g)$ is equal to the minimum of the distances between the players over all points in time. Considering all possible valid walks, we obtain the value of the strong (edge) vertex span. Throughout the paper we will use the language of walks to study the values of different span.

For the active and lazy movement rules the definitions are similar. To compensate for the rules change, additional properties of weak homomorphisms are required. The general idea remains the same as in the strong span variant.

\begin{definition}
Let $f,g:G\rightarrow H$ be weak homomorphisms. We say that $f$ and $g$ are \emph{aligned} weak homomorphisms, if for any $uv\in E(G)$,
\begin{displaymath}
f(u)f(v)\in E(H)    \Longleftrightarrow     g(u)g(v)\in E(H).
\end{displaymath}
\end{definition}

\begin{definition}\label{def:directSpans}
Let $H$ be a connected graph. 
Define 
\begin{align*}
  \deSpan{H} = \max \{ m_P(f,g)  \mid &f,g: P \rightarrow H \text{ are edge surjective aligned weak homomorphisms} \\ &\text{and $P$ is a path} \}.  
\end{align*}

We call $\deSpan{H}$ the \emph{direct edge span} of the graph $H$.

Define 
\begin{align*}
     \dvSpan{H} = \max \{ m_P(f,g)  \mid &f,g: P \rightarrow H \text{ are surjective aligned weak homomorphisms} \\ &\text{and $P$ is a path} \}. 
\end{align*}
 
We call $\dvSpan{H}$ the \emph{direct vertex span} of the graph $H$.
\end{definition}

\begin{definition}
Let $f,g:G\rightarrow H$ be weak homomorphisms. We say that $f$ and $g$ are \emph{opposite} weak homomorphisms, if for any $uv\in E(G)$,
\begin{displaymath}
f(u)f(v)\in E(H)    \Leftrightarrow     g(u)=g(v).
\end{displaymath}
\end{definition}

\begin{definition}\label{def:cartesianSpans}
Let $H$ be a connected graph. 
Define 
\begin{align*}
  \ceSpan{H} = \max \{ m_P(f,g)  \mid &f,g: P \rightarrow H \text{ are edge surjective opposite weak homomorphisms} \\ &\text{and $P$ is a path} \}.  
\end{align*}
     
We call $\ceSpan{H}$ the \emph{Cartesian edge span} of the graph $H$.

Define 
\begin{align*}
     \cvSpan{H} = \max \{ m_P(f,g)  \mid &f,g: P \rightarrow H \text{ are surjective opposite weak homomorphisms} \\ &\text{and $P$ is a path} \}. 
\end{align*}
 
We call $\cvSpan{H}$ the \emph{Cartesian vertex span} of the graph $H$.
\end{definition}

Let $G$ be a connected graph. We use $W: w_0, \ldots, w_k$ to denote a walk $W$ on $G$, i.e. $w_i w_{i+1} \in E(G)$ or $w_i = w_{i+1}$ for any $i\in\{0, \ldots, k-1\}$. A walk $W: w_0, \ldots, w_k$  is called a \emph{surjective vertex walk} if $\{w_i \mid i\in\{0, \ldots, k\}\} = V(G)$. Two surjective vertex walks $W_1: u_0, \ldots, u_k$ and $W_2: v_0, \ldots, v_k$ which achieve the strong vertex span, i.e. $d(u_i, v_i)\geq \svSpan{G}$ for any $i\in\{0, \ldots, k\}$, are called \emph{optimal strong vertex walks}. Analogously, \emph{optimal direct vertex} and \emph{optimal Cartesian vertex walks} are defined. A surjective vertex walk is a \emph{surjective edge walk} if  $\{w_i w_{i+1} \mid i\in\{0, \ldots, k-1\} \text{ and } w_i\not=w_{i+1}\} = E(G)$. Two surjective edge walks $W_1: u_0, \ldots, u_k$ and $W_2: v_0, \ldots, v_k$ which achieve the strong edge span, i.e. $d(u_i, v_i)\geq \seSpan{G}$ for any $i\in\{0, \ldots, k\}$, are called \emph{optimal strong edge walks}. Analogously, \emph{optimal direct edge} and \emph{optimal Cartesian edge walks} are defined. 

\begin{remark}\label{assumeAnyStart}
Let $G$ be a connected graph and $W_A: a_1, a_2, \ldots, a_k$ and $W_B: b_1, b_2, \ldots, b_k$ optimal strong (Cartesian or direct) vertex (edge) walks on $G$. For arbitrary vertices $u,v \in V(G)$
there exist $i, j \in \{1, \ldots, k\}$ such that $u=a_i$ and $v=a_j$. Then the walks $W'_A: a_i, a_{i-1}, \ldots, a_1, a_2,$ $ \ldots, a_k, a_{k-1}, \ldots, a_{j}$ and $W'_B: b_i, b_{i-1}, \ldots, b_1, b_2, \ldots, b_k, b_{k-1}, \ldots, b_j$ are also optimal strong (Cartesian or direct) vertex (edge) walks. So we can always assume that one player starts an optimal walk in a specific vertex and ends it in a specific vertex.
\end{remark}

\section{Bounds for different types of spans}\label{sec:bounds}

In \cite{BaTa23} for a simple graph $G$ the following bounds are given for $\star \in \{\BoxProduct, \times, \boxtimes\}$: $0 \leq \sigma^\star_E(G) \leq \sigma^\star_V(G) \leq \rad(G)$. For any $\star \in \{\BoxProduct, \times, \boxtimes\}$ graphs $G$ with  $\sigma^\star_E(G) = \sigma^\star_V(G)$ exist \cite{BaTa23}. There also exist graphs $G$ for which  $\sigma^\star_E(G) < \sigma^\star_V(G)$, as shown in the following example. The graph $G$ depicted in Figure~\ref{fig:VvsE} has $\sigma^\star_V(G)=2$ and $\sigma^\star_E(G)=1$. Note that Alice (or Bob) cannot traverse edge $e$ end keep distance 2 from each other. 

In this section we show that for any type of span the edge and the vertex variant can differ by at most 1, moreover, the graph in Figure~\ref{fig:VvsE} is an example of a graph where the difference between the edge and the vertex span is exactly 1. For any type of span we also improve the lower bound of the vertex variant with respect to the girth of the corresponding graph.

\begin{figure}[ht!]
	\begin{center}
		\begin{tikzpicture}[style=thick,x=1cm,y=1cm]
    		\def\vr{2.5pt} 
    		\draw (0,0)--(0,3);
    
            \foreach \x in {0,1,2,3}{
                \draw (0,\x)--(-1,1.5);
            }
     
    		\draw (1,1.5)--(0,1);
            \draw (1,1.5)--(0,2);
    
    		\draw (0,0) [fill=black] circle (\vr);
    		\draw (0,1) [fill=black] circle (\vr);
    		\draw (0,2) [fill=black] circle (\vr);
    		\draw (0,3) [fill=black] circle (\vr);
    		\draw (1,1.5) [fill=black] circle (\vr);
            \draw (-1,1.5) [fill=black] circle (\vr);
            \draw (0.1,1.5) node {$e$};
		\end{tikzpicture}
	\end{center}
	\caption{A graph $G$ with $\svSpan{G}=\seSpan{G}+1$}
	\label{fig:VvsE}
\end{figure}

\begin{theorem}
    For any connected graph $G$ the following hold true:
    \begin{enumerate}[(i)]
        \item $\svSpan{G} - \seSpan{G} \leq 1$,
        \item $\cvSpan{G} - \ceSpan{G} \leq 1$, and
        \item $\dvSpan{G} - \deSpan{G} \leq 1$.
    \end{enumerate}
    
\end{theorem}

\begin{proof}
\emph{(i)} Let Alice and Bob use two optimal strong vertex walks and let $e=uv$ be an arbitrary edge of $G$. At some point in time Alice (Bob) reaches the vertex $u$, with the distance at least $\svSpan{G}$ from Bob (Alice). Alice (Bob) can move to $v$ and back, while Bob (Alice) does not move. Therefore their distance when Alice (Bob) is at vertex $v$ is at least $\svSpan{G}-1$. Thus, Alice and Bob can visit all the edges of $G$ while being at distance at least $\svSpan{G}-1$.

\emph{(ii)} The same reasoning as in \emph{(i)} can be used for this case, starting with two optimal Cartesian vertex walks.

\emph{(iii)} Let $\dvSpan{G}=k$ and let $W: u_0, \ldots, u_t$, $W': v_0, \ldots, v_t$ be two optimal direct vertex walks for Alice and Bob. Hence, $d_G(u_i,v_i) \geq k$ for any $i \in \{0, \ldots, t\}$, and  $u_i u_{i+1}, v_i v_{i+1} \in E(G)$ for any $i \in \{0, \ldots, t-1\}$. Note also that $d_G(u_i,v_{i+1})\geq k-1$ and $d_G(v_i,u_{i+1}) \geq k-1$  for any $i \in \{0, \ldots, t-1\}$, even more for any neighbor $x$ of $u_i$ and any neighbor $y$ of $v_i$ it holds that $d_G(u_i,y) \geq k-1$ and $d_G(v_i,x) \geq k-1$.

Let Alice's walk be $U: u_0, u_1, u_0, u_1, u_2, \ldots, u_t$ and Bob's walk be $U': v_1, v_0, v_1, v_2, \ldots, v_t, v_{t-1}$. We extend these two walks (through all the vertices of $G$) to walks through all the edges of $G$ in the following way. For any edge $e=xx' \in E(G)\setminus E(W)$ there exists $i \in \{0, 1, \ldots, t\}$ such that $x=u_i$. Extend $U$ and $U'$ as follows: $U: u_0, u_1, u_0, u_1, \ldots, u_i, x', u_i, u_{i+1}, \ldots, u_t$ and $U': v_1, v_0, v_1, v_2, \ldots, v_{i+1}, v_i, v_{i+1}, \ldots, v_t, v_{t-1}$. In the same way, if $e=yy' \in E(G)\setminus E(W')$, then there exists $j\in \{2, \ldots, t\}$ (note that the  edge $v_0 v_1$ is already in $W'$ thus if one of $y$ or $y'$ equals $v_0$ or $v_1$ the other vertex is $v_j$ for some $j \geq 2$) such that $y=v_j$. Then $v_j, v_{j+1}$ part of $U'$ can be extended to $v_j, y', v_j, v_{j+1}$ and similarly $u_{j-1}, u_j$ part of $U$ by $u_{j-1}, u_j, u_{j-1}, u_j$. Since for any neighbour $x$ of $u_i$ and any neighbour $y$ of $v_i$ it holds that $d_G(u_i,y) \geq k-1$ and $d_G(v_i,x) \geq k-1$, Alice and Bob can stay at distance at least $k-1$ if they follow the extended $U$ and $U'$, respectively.
\end{proof}


\begin{theorem}
    Let $G$ be a connected graph with at least one cycle and $\girth(G)$ be the girth of the graph $G$. 
    Then the following hold true:
    \begin{enumerate}[(i)]
        \item $\svSpan{G} \geq \left\lfloor \frac{\girth(G)}{2} \right\rfloor$,
        \item $\cvSpan{G} \geq \left\lfloor \frac{\girth(G)}{2} \right\rfloor - 1$, and
        \item $\dvSpan{G} \geq \left\lfloor \frac{\girth(G)}{2} \right\rfloor - 1$. 
    \end{enumerate}
\end{theorem}

\begin{proof}
Let $C$ be a shortest cycle of $G$ and let $a_t$ and $b_t$ denote Alice's and Bob's positions at the time $t$, respectively. 

\emph{(i)} Let $d = \left\lfloor \frac{\girth(G)}{2} \right\rfloor$ and let Alice and Bob start their walks on antipodal vertices of $C$. They can clearly traverse all vertices of $C$ by both moving on $C$ at the same time in the same direction and keep the distance at least $d$. We claim that while Alice visits the vertices of the graph that are not on the cycle $C$, Bob can maintain the distance at least $d$ by only moving on the cycle $C$. Toward contradiction suppose, that up to some point in time, say $t$, Alice and Bob were able to maintain the distance at least $d$, and next, Alice wants to move from $a_t$ to a vertex $a_{t+1}$ that is at distance $d-1$ from $b_t$, and at most $d-1$ from both neighbours of $b_t$ on $C$, say $u_1$ and $u_2$ (otherwise Bob can move to one of these neighbours). But then any shortest $a_{t+1}, b_t$-path, $a_{t+1}, u_1$-path and the edge $b_t u_1$ generate a cycle of length at most $2d-1= 2 \left\lfloor \frac{\girth(G)}{2} \right\rfloor - 1 \leq \girth(G) - 1$. This contradicts the fact that the shortest cycle is of length $\girth(G)$.

\emph{(ii)} Let $d = \left\lfloor \frac{\girth(G)}{2} \right\rfloor - 1$ and let Alice and Bob start their walks on antipodal vertices of $C$. They can clearly traverse all vertices of $C$ by both moving on $C$ in the same direction while taking turns moving and keep the distance at least $d$. Note, when Alice (Bob) moves, then Bob (Alice) stays at the current position. Similarly to (i), we claim that while Alice visits the vertices of the graph that are not on the cycle $C$, Bob can maintain the distance at least $d$ by only moving on the cycle $C$. Toward contradiction suppose, that up to some point in time, say $t$, Alice and Bob were able to maintain the distance at least $d$ while obeying the movement rules and Bob only moved on $C$. Suppose next, Alice wants to move to a neighbour of $a_t$, say $a'$, and by doing so cannot keep the distance at least $d$ from Bob, i.e. $d_G(b_t, a')=d-1$, and Bob cannot move (before Alice moves in this turn) to any of the two neighbours of $b_t$ on $C$, say $u_1$ and $u_2$, and maintain distance at least $d$ with Alice. Thus, $d_G(a_t, u_1)=d_G(a_t, u_2)=d-1$. This implies that any shortest $a_t, u_1$-path, $a_t, u_2$-path and the edges $b_t u_1$ and $b_t u_2$ generate a cycle of length at most $(d-1)+(d-1)+2 = 2d = 2 \left\lfloor \frac{\girth(G)}{2} \right\rfloor -2 \leq 2 \frac{\girth(G)}{2} - 2 \leq \girth(G) - 2$, again a contradiction to $C$ being a shortest cycle of $G$.

\emph{(iii)} Note, at each point in time Alice and Bob must both move to an adjacent vertex. Again let Alice and Bob start their walks on antipodal vertices of $C$. They can clearly traverse all vertices of $C$ by moving on $C$ at the same time in the same direction and keep the distance at least $d$. We claim that while Alice visits the vertices of the graph that are not on the cycle $C$, Bob can maintain the distance at least $d$ by only moving on the cycle $C$. Toward contradiction suppose, that up to some point in time, say $t$, Alice and Bob were able to maintain the distance at least $d$ while Bob was moving only on $C$. Next, Alice wants to move from $a_t$ to a vertex $a_{t+1}$ that is at distance at most at most $d-1$ from both neighbours of $b_t$ on $C$, say $u_1$ and $u_2$. This implies that any shortest $a_{t+1}, u_1$-path, $a_{t+1}, u_2$-path and the edges $b_t u_1$ and $b_t u_2$ generate a cycle of length at most $(d-1)+(d-1)+2 = 2d = 2 \left\lfloor \frac{\girth(G)}{2} \right\rfloor -2 \leq 2 \frac{\girth(G)}{2} - 2 \leq \girth(G) - 2$, again a contradiction to $C$ being a shortest cycle of $G$.
\end{proof}

\section{On graphs with the strong vertex span 1}\label{sec:span1}

In the first part of this section we show some necessary conditions for graphs with the strong vertex span equal to 1. For this we will often use the following straight forward result.

\begin{lemma}\label{l:minCutNeighboursEverywhere}
    If $G$ is a connected graph and $S$ its minimal cut set, then every vertex of $S$ has a neighbour in each component of $G-S$.
\end{lemma}

\begin{proof}
Let $G_1, G_2, \ldots, G_k$, where $k\geq 2$, be the components of $G-S$. Towards contradiction suppose that there exists $x \in S$ with no neighbour in $G_i$ for some $i \in \{1,\ldots , k\}$. Set $S'= S \setminus \{x\}$. Clearly, $S' \subseteq S$ and $G-S'$ has at least two components: $G_i$ and the component containing $x$, thus $S'$ is a cut set, contradicting the fact that $S$ is a minimal cut set.
\end{proof}

Let $G$ be a connected graph and $S\subseteq V(G)$. An $S$-lobe of $G$ is an induced subgraph of $G$ whose vertex set consists of $S$ and all the vertices of a component of $G-S$. If $S$ is not a cut set, then there is exactly one $S$-lobe, namely the graph $G$. If $S$ is a cut set of a connected graph $G$ and $G_1, \ldots, G_k$ the components of $G-S$, then for any $i \in \{1, \ldots, k\}$ the subgraph of $G$ induced by $V(G_i) \cup S$ is an $S$-lobe of $G$. 


\begin{lemma}\label{l:SlobeSpan=1}
    If $G$ is a connected graph with $\svSpan{G}=1$, then for any minimal cut set $S$ all $S$-lobes have the strong vertex span 1. Moreover, if $H_1, \ldots, H_k$ are $S$-lobes of $G$, then $\svSpan{G[V(H_{i_1}) \cup \ldots \cup V(H_{i_j})]}=1$ for any $\{i_1,\ldots ,i_j\} \subseteq \{1, \ldots, k\}$.
\end{lemma}

\begin{proof}
    Let $S$ be a minimal cut set of $G$. Let $G_1, G_2, \ldots, G_k$ be the components of $G-S$ and let the corresponding $S$-lobes be $H_1, \ldots, H_k$. Towards contradiction assume that there exists $i \in \{1, \ldots, k\}$ such that $\svSpan{H_i}>1$. 
    
    Let Alice and Bob both start their walks on $G$ in $H_i$ by visiting all vertices of $H_i$ while keeping distance at least 2 from each other (this is possible as $\svSpan{H_i}>1$). Let $W: u_0, \ldots, u_m$, $W':v_0, \ldots, v_m$ be optimal strong walks on $H_i$ for Alice and Bob, respectively. We will extend both these walks on walks through all the vertices of $G$. Since $S \subseteq V(H_i)$, there is exists $t$ such that $u_t \in S$. When Alice is in $u_t$, Bob is in $v_t$ and they are at distance at least 2. By Lemma \ref{l:minCutNeighboursEverywhere}, $u_t$ has a neighbour in every component of $G-S$. Thus, for all $j\not = i$ Alice can visit each component $G_j$ from $u_t$, traverse all the vertices of $G_j$ and backtrack to $u_t$. During this Bob uses one of the following strategies: if $v_t \not \in S$ (therefore $v_t \in V(G_i)$), then Bob remains in $v_t$ and the distance between him and Alice is at least 2 for each of the new Alice's steps. If $v_t \in S$, then by Lemma \ref{l:minCutNeighboursEverywhere} there is a vertex $w \in V(G_i)$ such that $v_t w \in E(G)$. For each of the new Alice's steps Bob moves to or stays at $w$, except when Alice is in $u_t$, then Bob moves to $v_t$. In this manner Alice can visit all the vertices of $G$ and keep the distance at least 2 from Bob. After that they swap strategies and Bob can visit all vertices of $G$ while keeping the distance to Alice at least 2. Hence, $\svSpan{G} \geq 2$, a contradiction.

    A similar line of thought can be used to prove the moreover part of the lemma.
\end{proof}

\begin{lemma}\label{l:clique}
    If $G$ is a connected graph with $\Delta(G) < |V(G)|-1$ and $\svSpan{G}=1$, then any minimal cut set $S$ of $G$ is a clique.
\end{lemma}
\begin{proof}
    Let $S$ be a minimal cut set of $G$ and $G_1, \ldots, G_k$ the  components of $G-S$. Towards contradiction suppose that $S$ is not a clique, thus there exist two distinct vertices of $S$, say $s_1$ and $s_2$, such that $s_1 s_2 \not \in E(G)$. We will show that Alice can visit all the vertices of $G$ while keeping the distance at least 2 to Bob at all times. After that, Bob can use Alice's strategy to do the same. In what follows, if not stated otherwise, while one player is moving, the other remains in the same position. 

    Let Alice and Bob start in two distinct components of $G$, $G_i$ and $G_j$, respectively. Note first, Alice can visit all the vertices of $G_i$ while moving only on vertices of $G_i$ ($G_i$ is connected) and if Bob stays in $G_j$ (e.g. Bob does not move), the distance to Alice is always at least 2 during this. 
    
    Next, we will show that Alice and Bob can move from $G_i$ and $G_j$, respectively, to any other two distinct components, say $G_{i'}$ and $G_{j'}$, respectively, and still maintain distance at least 2. By Lemma \ref{l:minCutNeighboursEverywhere}, every $s\in S$ has a neighbour in every component of $G-S$. For each $s \in S$ and each $y \in \{1, \ldots, k\}$ denote by $s^y$ one such neighbour of $s$ in $G_y$. Alice and Bob can swap components in the following way. Alice moves within $G_i$ to $s_1^i$, and Bob moves within $G_j$ to $s_2^j$. When they are at their destinations, they both move at the same time, Alice to $s_1$ and Bob to $s_2$. Since  $s_1 s_2 \not \in E(G)$, the distance between them is still at least 2. After this, again at the same time, Alice moves to $s_1^{i'}$ and Bob to $s_2^{j'}$. Doing this they both end up in the desired components, Alice in $G_{i'}$ and Bob in $G_{j'}$. Note, by doing this, Alice also visits $s_1$. Moreover, if we swap the roles of $s_1$ and $s_2$, Alice can also visit $s_2$. Combining the described strategies, Alice can visit all the vertices in $G - S$ and both $s_1$ and $s_2$ and always be at distance at least 2 from Bob.

    What remains is to show that Alice can also visit any vertex of $S\setminus\{s_1, s_2\}$ and still be at distance at least 2 from Bob. W.l.o.g. we can assume that at this point in time Alice is in $G_i$ and Bob is in $G_j$, where $i\not= j$ (if not, then they can backtrack all their moves and end up in these starting components).
    Let $s \in S\setminus\{s_1, s_2\}$ be arbitrary. Since  $\Delta(G) < |V(G)| - 1$ there is a vertex $u \in V(G)$, such that $d(s, u) \geq 2$. If $u \in S$, then Alice moves to $s^i$ and Bob to $u^j$, after that at the same time Alice moves to $s$ and Bob to $u$, and then at the same time they return to their previous position. If $u \not \in S$, then $u$ is in some component of $G-S$, say $G_{j'}$. Then they swap components with accordance to the previously described strategy, such that Bob ends up in $G_{j'}$ and Alice in any other component, say $G_{i'}$. Next, Bob moves to $u$, Alice moves to $s^{i'}$ and then to $s$ and back. In both cases Alice was able to visit $s$ and be at distance at least 2 from Bob at all times. 

    Following the described strategies Alice can visit every vertex of $G$ and keep at distance at least 2 from Bob at all points in time. Now, they swap roles and Bob can use the same strategies to visit all vertices of $G$ and also keep at distance at least 2 from Alice. This implies that $\svSpan{G} > 1$, a contradiction.
\end{proof}



\begin{lemma}
    Let $G$ be a graph with $\svSpan{G}=1$, $\Delta(G) < |V(G)|-1$, $S$ a minimal cut set of $G$ and $G_1, \ldots, G_k$ the components of $G-S$. Then $G[S \cup V(G_i)] \cong G[S] \vee G_i$ holds for all except at most two indices $i \in \{1, \ldots, k\}$. 
\end{lemma}

\begin{proof}
    Towards contradiction suppose that there are at least three components of $G-S$, say $G_1, G_2,$ and $G_3$, such that $G[S \cup V(G_i)] \ncong G[S] \vee G_i$ for all $i\in\{1,2,3\}$. Therefore there exist $g_1 \in V(G_1), g_2 \in V(G_2), g_3 \in V(G_3),$ and $s_1, s_2, s_3 \in S$ such that none of $g_1s_1, g_2s_2, g_3s_3$ are edges of $G$. Similarly to the proof of Lemma \ref{l:clique} we will show that Alice can visit all the vertices of $G$ while keeping the distance at least 2 to Bob at all times. After that, Bob can use Alice's strategy to do the same. Again, if not stated otherwise, while one player is moving, the other remains in the same position. 

    Note first, whenever Alice and Bob are in distinct components of $G-S$, then Alice can visit all the vertices of her current component while moving only on vertices of that component and if Bob does not move, the distance to Alice is always at least 2 during this. 

    Let Alice and Bob start their walks in $g_1$ and $g_2$, respectively. First, we will describe, how Alice can move from $G_1$ to any other component $G_i$, where $i>2$ and maintain distance at least 2 from Bob. Since Alice can always backtrack her steps, we may assume she always does this from $G_1$. By Lemma \ref{l:minCutNeighboursEverywhere}, every $s\in S$ has a neighbour in every component of $G-S$. For each $s \in S$ and each $y \in \{1, \ldots, k\}$ denote by $s^y$ one such neighbour of $s$ in $G_y$. Alice moves to $s_2^1$, Bob moves to $g_2$, then Alice moves to $s_2$ (note, $g_2 s_2 \not \in E(G)$) and then to $s_2^i$ (thus she reaches the component $G_i$). 
    
    Second, if Alice wants to move to $G_2$ (Bob's component), she can do this and maintain distance at least 2 from Bob by doing as follows. Alice moves to $g_1$, Bob moves to $s_1^2$, then Bob moves to $s_1$ (note, $g_1 s_1 \not \in E(G)$) and then to $s_1^3$. Now, Alice is in $G_1$ and Bob in  $G_3$. Next Alice moves to $s_3^1$ and Bob to $g_3$, then Alice to $s_3$ (note again, $g_3 s_3 \not \in E(G)$), and finally she moves to $s_3^2$ and reaches $G_2$. Clearly, the distance between Alice and Bob during these swaps was always at least 2. Note also, Alice visited $s_2$ and $s_3$ during this. A similar strategy can be used for Alice to visit $s_1$, as well and still be at distance at least 2 from Bob. So Alice can use the described strategies to visit all the vertices of $G_i$, for all $i\in \{1, 2, \ldots, k\}$ and $s_1, s_2, s_3$ and always be at distance at least 2 from Bob.

    Finally, let $s\in S\setminus\{s_1, s_2, s_3\}$ be arbitrary. We must show that Alice can visit $s$ and still maintain the distance at least 2 from Bob at all points in time. Since $\Delta(G) < |V(G)| - 1$ there is a vertex $u \in V(G)$, such that $d(s, u) \geq 2$. Moreover, by Lemma \ref{l:clique} it follows that $u \not \in S$, i.e. $u \in G_{j'}$, for some ${j'} \in \{1, 2, \ldots, k\}$. Alice and Bob can use previous strategies, so that Bob ends up in $G_{j'}$ and Alice in some component $G_{i'}$, where ${i'}\not = {j'}$, and maintain the distance at least 2. Now Alice moves to $s^{i'}$, Bob moves to $u$, then Alice moves to $s$ and back. Since $su \not \in E(G)$ Alice can visit $s$ while maintaining the distance at least 2 from Bob.

    It follows that Alice can visit every vertex of $G$ and keep at distance at least 2 from Bob at all points in time. Now, they swap roles and Bob can use the same strategies to visit all vertices of $G$ and also keep at distance at least 2 from Alice. This implies that $\svSpan{G} > 1$, a contradiction.
\end{proof}

In the second part of this section we show some sufficient properties for a graph to have the strong vertex span 1. An \emph{interval representation} of a graph is a family of intervals of the real line assigned to vertices so that vertices are adjacent if and only if the corresponding intervals intersect. A graph is an \emph{interval graph} if it has an interval representation. One may assume that all intervals in any interval representation of an interval graph are closed, have positive length, and no intervals agree on any endpoint. 

Let $G$ be an interval graph. Given an interval representation of $G$, for all $v\in V(G)$ let $I_v=[l_v, r_v]$ be the interval corresponding to $v$. A vertex $v$ of $G$ is called an \emph{end-vertex} if there exists an interval representation of $G$ such that $l_v = \min\{ l_u \mid u\in V(G)\}$ or $r_v = \max\{ r_u \mid u\in V(G)\}$, the corresponding interval $I_v$ is called an \emph{end-interval} \cite{Gim88}. A maximal clique of  $G$ is an \emph{end-clique} if it contains an end vertex of $G$. It is easy to see that an end-vertex $v$ of $G$ is always simplicial, i.e. $N[v]$ is a clique. 

An independent set of three vertices such that each pair is joined by a path that avoids the neighbourhood of the third is called an \emph{asteroidal triple}. Interval graphs can be characterised as asteroidal triple-free chordal graphs \cite{bls-99}.

\begin{theorem}\label{th:interval}
    If $G$ is a non-trivial connected interval graph, then $\svSpan{G}=1$.
\end{theorem}

\begin{proof}
    Choose an arbitrary interval representation of $G$ and for all $v\in V(G)$ let $I_v=[l_v, r_v]$ be the interval corresponding to $v$.
    Order the vertices of $G$ with respect to the left end-points in their interval representation, say $v_1, v_2, \ldots, v_n$ where $l_{v_1} \leq l_{v_2} \leq \ldots \leq l_{v_n}$.

    By Remark \ref{assumeAnyStart} assume Alice starts her optimal strong vertex walk in $a_1=v_1$. Since graph is non-trivial, it follows that $\svSpan{G}\geq 1$. Let Bob be at any other vertex of $G$, say $b_1 \not= a_1$, in his corresponding strong vertex optimal walk. If $d(a_1, b_1)=1$ then $\svSpan{G}\leq 1$ and this part is done. 
    
    Assume now, $d(a_1, b_1)>1$. Therefore $r_{a_1} < l_{b_1}$. Since Alice visits all vertices of $G$, at some point in time, say $t+1$, it must hold that $l_{b_{t+1}} \leq l_{a_{t+1}}$, and for all $t' < t+1$ it holds true that  $l_{b_{t'}} > l_{a_{t'}}$. If $l_{b_t} \leq r_{a_t}$, then $d(a_t, b_t)=1$ and $\svSpan{G}\leq 1$, again we are done. Otherwise, since $l_{a_t} < l_{b_t}$ and  $l_{b_{t+1}} \leq l_{a_{t+1}}$, $a_{t+1} \in N[a_t]$ and $b_{t+1} \in N[b_t]$ it follows that $I_{a_{t+1}} \cap I_{b_{t+1}} \not= \emptyset$. Which means that $d(a_{t+1}, b_{t+1})=1$, again the assertion follows and the proof is concluded.
\end{proof}

If $T$ is a non-trivial tree, then $\svSpan{T}=1$ is also a sufficient condition for $T$ to be an interval graph. A \emph{subdivided star} $S_{1,n}$ is a tree obtained from the star $K_{1,n}$ by subdividing each edge of $K_{1,n}$ exactly once. A tree $T$ is an interval graph if and only if it does not contain $S_{1,3}$ as an induced subgraph~\cite{mcmc-99}. For more details on interval graphs see \cite{bls-99, mcmc-99}.  

\begin{theorem}
    Let $T$ be a non-trivial tree. Then $T$ is an interval graph if and only if $\svSpan{T}=1$.
\end{theorem}
\begin{proof}
    If $T$ is an interval graph, then $\svSpan{T}=1$ by Theorem~\ref{th:interval}. For the converse let $\svSpan{T}=1$. For the purpose of contradiction assume that $T$ is not an interval graph. Thus $T$ contains $S_{1,3}$ as an induced subgraph. Denote vertices of an arbitrary $S_{1,3}$ in $T$ by $x$ (the vertex of degree 3 in $S_{1,3}$), $a_1,a_2,a_3$ (the vertices of degree 2 in $S_{1,3}$) and $b_1,b_2,b_3$ (the remaining three vertices) and choose the notation such that $a_i b_i \in E(T)$ for any $i\in \{1, 2, 3\}$. Also, for any $i \in \{1, 2, 3\}$ denote the component of $T-x$ that contains $a_i$ by $T_{a_i}$. Let Alice and Bob start their walks on $T$ in $b_1$ and $b_2$, respectively. Then Bob can visit all vertices of $T-T_{a_1}$ and return to $b_2$, while Alice stays in $b_1$. At this point, the distance between Alice and Bob is at least 2. After that Bob stays in $b_2$, while Alice visits all vertices of $T_{a_1}$, returns to $b_1$ and then moves to $b_3$ using the shortest $b_1,b_3$-path. Note that Alice and Bob still keep the distance at least 2 at all points in time. Finally, Bob moves from $b_2$ to $a_1$, visits all vertices of $T_{a_1}$ and returns to $b_1$, while Alice stays in $b_3$. After that, Bob stays in $b_1$, while Alice visits all the vertices of $T-T_{a_1}$. Since the distance between Alice and Bob is at least 2 at each point in time during their walks, we get a contradiction with $\svSpan{T}=1$.
\end{proof}

The result from trees unfortunately cannot be generalised to arbitrary graphs. There exist graphs $G$ with large induced cycles having $\svSpan{G}=1$ and also graphs $G$ containing asteroidal triples having $\svSpan{G}=1$. Clearly if $G$ is obtained from an arbitrary graph $G'$ by adding a new vertex $x$ and connecting $x$ to all vertices of $G'$ with an edge, then $G$ is a graph with $\svSpan{G}=1$. Thus any non-trivial graph $G$ that contains a universal vertex (a vertex of degree $|V(G)|-1$), has $\rad(G)=1$ and therefore $\svSpan{G}=1$. But there also exist non-interval graphs $G$ (either non-chordal or such with asteroidal triples) with $\Delta(G) < |V(G)|-1$ and $\svSpan{G}=1$ (see Figure~\ref{fig:non-chordal}).

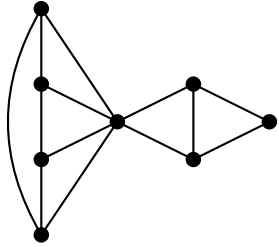
\begin{figure}[ht!]
	\begin{center}
		\begin{tikzpicture}[style=thick,x=1cm,y=1cm]
    		\def\vr{2.5pt} 
    		\draw (0,0)--(0,3);
    
            \foreach \x in {0,1,2,3}{
                    \draw (0,\x)--(1,1.5);
            }
     
    		\draw (1,1.5)--(2,1);
            \draw (1,1.5)--(2,2);
    		\draw (2,1)--(2,2);
    		\draw (0,0) to [bend left=30] (0,3);
            \draw (3,1.5)--(2,2);
            \draw (3,1.5)--(2,1);
      
    		\draw (0,0) [fill=black] circle (\vr);
    		\draw (0,1) [fill=black] circle (\vr);
    		\draw (0,2) [fill=black] circle (\vr);
    		\draw (0,3) [fill=black] circle (\vr);
    		\draw (1,1.5) [fill=black] circle (\vr);
            \draw (2,1) [fill=black] circle (\vr);
            \draw (2,2) [fill=black] circle (\vr);
            \draw (3,1.5) [fill=black] circle (\vr);
		\end{tikzpicture}
	\end{center}
	\caption{Non-chordal graph $G$ with $\svSpan{G}=1$}
	\label{fig:non-chordal}
\end{figure}

In the next two theorems we provide a construction of an infinite family of graphs, say ${\mathcal{F}}$, such that it contains non-trivial interval graphs and any graph $G \in {\mathcal{F}}$ has $\svSpan{G}=1$.

\begin{definition}\label{def:aug}
Let $G$ and $H$ be connected vertex disjoint graphs and $S\subseteq V(G)$. The augmentation of $G$ with respect to $S$ and $H$ is the graph $\aug(G,S,H)$ with the vertex set $V(G) \cup V(H)$, and the edge set 
\[
    E(G) \cup E(H) \cup \{ sh \mid s \in S \land h \in V(H) \}.
\]
\end{definition}

\begin{theorem}\label{thm:naendkliko}
    Let $G$ be an interval graph, $K$ its end-clique and $H$ an arbitrary graph disjoint with $G$. If $G'=\aug(G,K,H)$, then $\svSpan{G'}=1.$
\end{theorem}

\begin{proof}
    If $G$ is a complete graph, then $G'$ contains a universal vertex, thus $\rad(G')=1$, therefore $\svSpan{G'}=1$.
    Suppose now, $V(G)$ is not a clique, therefore $K\subset V(G)$. Let $x\in K$ be an end-simplicial vertex of $G$, i.e. $N_G[x] = K$. Towards contradiction suppose that $\svSpan{G'}\geq 2$. Let $W_A : a_1, \ldots, a_\ell$ and $W_B : b_1, \ldots, b_\ell$ be an optimal pair of strong vertex walks for Alice and Bob, respectively. Moreover, assume by Remark \ref{assumeAnyStart}, that $a_1 = a_\ell = x$. Since $W_A, W_B$ are optimal and $\svSpan{G'}\geq 2$, for any $i\in \{1, \ldots, \ell\}$ it holds true that $d_{G'}(a_i, b_i)\geq 2$. 

    We will transform these walks on $G'$ into walks through all vertices of $G$ while maintaining the distance at least 2 between Alice and Bob at all times. 
    
    First for Alice, we do the following procedure for all possible indices $i$ and $j$, where $i < j-1$, $a_i, a_j \in K$ and $a_k \in V(H)$ for any $k\in \{i+1, \ldots, j-1\}$. Note, for any $k \in \{i, i+1, \ldots, j\}$, $b_k \not \in V(H) \cup K$, hence $d_{G}(b_k, x) \geq 2$. Let $(a'_1, a'_2, \ldots, a'_i, a'_{i+1}, \ldots, a'_{j-1}, a'_j, a'_{j+1}, \ldots, a'_\ell)$, where $a'_k = x$ for all $k\in \{i+1, \ldots, j-1\}$ and $a'_k = a_k$ for all $k\in \{1, \ldots, \ell\} \setminus \{i+1, \ldots, j-1\}$. Therefore $d_{G}(a'_k, b_k)\geq 2$ for any $k\in\{1,\dots, \ell\}$. Denote by $W'_A$ the walk obtained by this procedure after there are no more vertices of $V(H)$ in the walk. In the same manner we can remove all vertices of $W_B$ that are in $V(H)$, denote by $W'_B$ the obtained walk. This means that $W'_A$ and $W'_B$ are walks through all the vertices of $G$ such that at any time the distance between Alice and Bob is at least 2, therefore $\svSpan{G}\geq 2$, a contradiction to Theorem \ref{th:interval}.
\end{proof}

\begin{theorem}\label{thm:naMinSep}
    Let $G$ be an interval graph, $K$ its minimal cut set and $H$ an arbitrary graph disjoint with $G$. If $G'=\aug(G,K,H)$, then $\svSpan{G'}=1.$
\end{theorem}

\begin{proof}
    If $G'$ contains a universal vertex, then $\rad(G')=1$, therefore $\svSpan{G'}=1$. 
    
    Suppose now, $G'$ has no universal vertices. Note, this also implies that $G$ has no universal vertices. If this is not the case, then since $K$ is a minimal cut set, it would have to contain every universal vertex, and therefore these vertices would also be universal in $G'$ by Definition \ref{def:aug}.
    
    By Theorem \ref{th:interval} we have $\svSpan{G}=1$, and since $G$ has no universal vertices, using Lemma \ref{l:clique} it follows that $K$ is a clique. Moreover, $K$ is a minimal cut set of $G$, therefore by Lemma \ref{l:minCutNeighboursEverywhere} every vertex of $K$ has a neighbour in each component of $G - K$.

    Towards contradiction suppose that $\svSpan{G'}>1$. Let $W_A: a_1, \ldots, a_k$ and $W_B: b_1, \ldots, b_k$ be two optimal strong walks for Alice and Bob, respectively, and by Remark \ref{assumeAnyStart} let  $a_1 \not \in V(H)$. Let $i$ and $j$, $i < j$, be the smallest integers such that $a_i \in K$, $a_\ell \in V(H)$ for every $ i < \ell < j$, and $a_j \in K$ or $j>k$. In other words, take the first sub-walk of $W_A$ such that Alice moves from $K$ to $H$, visits some vertices from $H$, and either stops in $H$ or returns to $K$. Note, if Alice is not already in $H$, she can only move to $H$ from a vertex in $K$. When Alice is in $a_i$, Bob is in $b_i \not \in K\cup V(H)$, because $d_{G'}(a_i, b_i)\geq \svSpan{G'} > 1$. Since $K$ is a cut set of $G$ it follows, that Bob is in some component of $G-K$, say $C_1$ and there exists at least one other component of $G-K$, say $C_2$. By Lemma \ref{l:minCutNeighboursEverywhere} there exist vertices $c_2^i, c_2^j \in V(C_2)$, such that $c_2^i \in N_G(a_i)$ and $c_2^j \in N_G(a_j)$.
    
    Let $W_{A'}$ and $W_{B'}$ be the walks obtained from $W_A$ and $W_B$, respectively, as follows. We start the change at the time $i$, with Alice in $a_i$ and Bob in $b_i$. While Bob moves through the vertices $b_{i+1}, b_{i+2}, \ldots, b_{j-1}$, Alice moves to and waits in $c_2^i$. Since they are in different components of $G-K$ at all times, they are also at distance at least 2 at all times. Now Bob waits in $b_{j-1}$, while Alice moves from $c_2^i$ to $c_2^j$ by only using vertices from $C_2$, thus maintaining the distance at least 2 from Bob. Finally, at the same time, Alice moves to $a_j$ and Bob moves to $b_j$ (if $j$ is a valid index, otherwise their walks end in the previous move). We can now repeat the same idea on the changed walks $W_{A'}$ and $W_{B'}$, as long as $W_{A'}$ contains vertices from $H$. After the procedure is finished, we have changed the original walks in such a way, that instead of visiting vertices of $H$ Alice walks on vertices outside of $H$ and only the vertices from $H$ were removed, while Bob visits the same set of vertices as in the original walk. Moreover, at all times they still keep the distance at least two. 

    Alice and Bob can now swap the roles and apply the same strategy for Bob. We end up with two walks, both through all the vertices of $G'-H$, i.e. $G$, while Alice and Bob are at the distance at least two at all times, thus $\svSpan{G}\geq 2$. A contradiction with the fact that $\svSpan{G}=1$, therefore $\svSpan{G'}\leq 1$. Since $G'$ is not the trivial graph, it follows that $\svSpan{G'}=1$.
\end{proof}

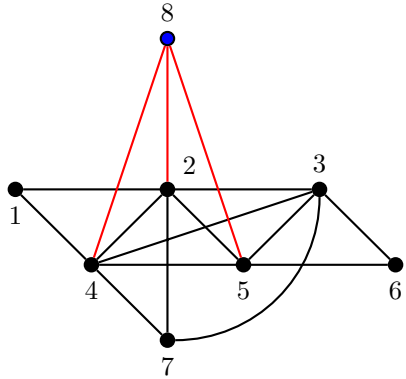
\begin{figure}[ht!]
	\begin{center}
 \begin{tikzpicture}[style=thick, main_node/.style={circle, fill=black,draw, inner sep=0pt, minimum size=5pt}]

\node[main_node] (1) at (-4, 2) [label=below:1]{};
\node[main_node] (2) at (-2, 2) [label=45:2]{};
\node[main_node] (3) at (0, 2) [label=above:3]{};
\node[main_node] (4) at (-3, 1) [label=below:4]{};
\node[main_node] (5) at (-1, 1) [label=below:5]{};
\node[main_node] (6) at (1, 1) [label=below:6]{};
\node[main_node] (7) at (-2, 0) [label=below:7]{};
\node[main_node,fill=blue] (8) at (-2, 4) [label=above:8]{};

\draw (1)--(3)--(6)--(4)--(1);
\draw (4)--(2)--(5)--(3);
\draw (3)--(4)--(7)--(2);
\draw (7) to[out=0,in=-90] (3);
\draw[red] (4)--(8)--(5);
\draw[red] (2)--(8);
\end{tikzpicture}
	\end{center}
	\caption{An augmentation $G$ of an interval graph with respect to a non-minimal separating clique and a trivial graph with $\svSpan{G}=2$}
	\label{fig:primerIGspan2}
\end{figure}

The results from Theorem \ref{thm:naendkliko} and \ref{thm:naMinSep} cannot be generalised to arbitrary cliques of interval graphs, not even separating cliques, as the graph in Figure \ref{fig:primerIGspan2} shows. In the figure the graph induced by the vertices $\{1, 2, \ldots, 7\}$, say $G$, is an interval graph which can be easily checked. The set $S=\{2, 4, 5\}$ is a clique, moreover it is a cut set of $G$; the graph $G-S$ has two components, induced by the sets $\{1\}$ and $\{3, 6, 7\}$, respectively.  Let $H$ be the trivial graph with the vertex set $\{8\}$. The graph shown in Figure \ref{fig:primerIGspan2} is the graph $\aug(G,S,H)$, where the graph $H$ is depicted with the blue colour, and the edges introduced by the augmentation procedure are drawn in red. By Theorem \ref{th:interval}, $\svSpan{G}=1$. Let $W_A=(1,2,8,4,7,3,6,5,6,6,6,5,8)$ and $W_B=(5,6,6,6,5,8,2,1,2,8,4,7,3)$ be Alice's and Bob's walks, respectively. Both of them visit each vertex of $\aug(G,S,H)$ at least once, moreover, at each step, the distance between Alice and Bob is 2, therefore $\svSpan{\aug(G,S,H)}\geq 2 \not = 1$.

We already know, that any graph $G$ is an induced subgraph of a graph $G'$ with $\svSpan{G'}=1$. Indeed an example of such graph $G'$ is a graph obtained from $G$ by adding a new vertex and connect it to all vertices of $G$. Then $\rad(G')=1$ and thus $\svSpan{G'}=1$. Theorems \ref{thm:naendkliko} and \ref{thm:naMinSep} also imply that an arbitrary graph $G$ is also an induced subgraph of a graph $G'$ with $\svSpan{G'}=1$ and arbitrary large radius.

\section{Minimum number of steps in optimal walks}\label{sec:minimNumSteps}

Given a graph $H$ and the movement rules, we are interested in finding the vertex span of $H$ and afterwards the optimal walks, with the respect to the movement rules, visiting all the vertices of $H$ keeping the maximal safety distance at all times. In order to find the optimal walks, we consider the product graph of $H$ with itself, the type of the product depending on the movement rules, and use the fact that optimal walks contain all the vertices of $H$ in both projections. Therefore, searching for the minimum number of steps in optimal walks is basically searching for the shortest walk in the product graph, that would visit all vertices of $H$ in both projections.

We have an exponential time algorithm that calculates the minimum number of steps in optimal walks, for any of the three movement rules. 

The algorithm works as follows: given graph $H$ and the movement rules, we first compute the corresponding span of $H$. With this in mind, we will use the existing algorithm by Banič and Taranenko given in~\cite{BaTa23}, which computes the chosen span of a given graph $H$ correctly in polynomial time considering any movement rule $R$ of the three possible: lazy, traditional and active, and address it as \verb|span(H, R)|. Denote by $G$ the graph obtained from the chosen product graph of $H$ with itself by removing all vertices $(u,v)$ such that $d_H(u,v) < \verb|span(H, R)|$. This notation is used throughout the remainder of the section. The algorithm in~\cite{BaTa23} implies that there is a connected component in $G$ such that it projects surjectively in both projections to $V(H)$. We call such a component a \emph{good component}. In each of the good components, we search for the shortest walk that covers all the vertices of $H$ in both projections by using the algorithm below which checks all possible walks. 
We then search for the minimum length over all good components.

We present the general concept of our algorithm in pseudo-code. First, we give auxiliary algorithms necessary for calculating the minimum number of steps in optimal walks.

\begin{algorithm}
\caption{goodComponents(G, H)}
\label{alg:goodC}
\KwIn{graph $G$, and the base graph $H$}
\KwOut{list of components that projects to $V(H)$ in both projections}
$gc = \emptyset$

\tcc{we use the existing function to find all the connected components of the graph}
$comp = listAllConnectedComponents(G)$

\ForEach{c in comp}
{
%
\If{c projects surjectively in both projections to $V(H)$}
{
$gc \gets gc \cup c$
}
}
\KwRet{gc}
\end{algorithm}
\begin{theorem}
Algorithm~\ref{alg:goodC} correctly identifies good components of the given graph $G$ surjectively projecting to $H$ in both projections in polynomial time.
\end{theorem}
\begin{proof}
Denote by $N=|V(G)|$. To obtain the list of connected components of $G$, in line $2$ of Algorithm~\ref{alg:goodC}, we use the function from \emph{networkx} package in Python. For each vertex $v$ in $G$, it checks whether it belongs to already visited set of vertices, and if not, then it applies the BFS algorithm from that vertex $v$ to obtain the component containing $v$. This algorithm in total runs in $O(N+E(G))=O(N^2)$ time. Therefore, the components can be found in polynomial time. Now, we need to check whether any component projects surjectively to $V(H)$ in both projections. This is checked in lines $3-5$ of the algorithm and can be checked in time $O(N)$. In total, this algorithm runs in time $O(N^2)$, which concludes the proof.
\end{proof}










\begin{algorithm}
\caption{minLenWalk(G, H, x, walk, visited, visitedE)}
\label{alg:mnLenPath}
\KwIn{Graph $G$, graph $H$ (the base graph), starting vertex $x$, given $walk$ and list of visited vertices $visited$, and list of visited edges $visitedE$}
\KwOut{shortest walk and its length}

$ neighbours = \texttt{ListAllNeighbours(x)}$

$mn \gets \infty$
 
$minWalk \gets walk$

\tcc{check whether all the neighbours of x are visited, and if they are, try using a new edge, to minimize the length of the walk}
\ForEach{node a in neighbours}
{
    \uIf{a $\not \in$ visited}
    {
        $walk \gets walk \cup \{a\}$

        $visited \gets visited \cup \{a\}$

        $visitedE \gets visitedE \cup \{(v, a)\}$

        \uIf{walk projects surjectively in both projections to $V(H)$}
        {
            \If{$length(walk) < mn$}
            {
                $minWalk = walk$

                $mn = length(walk)$
                
            }
        }
        \Else
        {
            $[ml, mw] = \texttt{minLenWalk(G, H, a, walk, visited, visitedE)}$

            \If{$ml < mn$}
            {
                $minWalk = mw$

                $mn = ml$
                
            }
        }

        remove $a$ from $walk$ and $visited$, and $(v, a)$ from $visitedE$
    }
     \Else
     {
        \If{(a,v) $\not \in$ visitedE}
         {
             $walk \gets walk \cup \{a\}$

             $visited \gets visited \cup \{a\}$

             $visitedE \gets visitedE \cup \{(a,v)\}$

            \uIf{walk projects surjectively in both projections to $V(H)$}
            {
                \If{$length(walk) < mn$}
                {
                    $minWalk = walk$

                    $mn = length(walk)$
                
                }
            }
            \Else
            {
                $[ml, mw] = \texttt{minLenWalk(G, H, a, walk, visited, visitedE)}$

                \If{$ml < mn$}
                {
                    $minWalk = mw$

                    $mn = ml$
                
                }
        }
        remove $a$ from $walk$ and $visited$ and $(a,v)$ from $visitedE$
        }
        \Else{
            \tcc{try going backwards along the edges of this walk until a new neighbour is found and continue the walk from there}
        }
    }
}
\KwRet{$mn, minWalk$}
\end{algorithm}

\begin{theorem}
    Algorithm~\ref{alg:mnLenPath} returns a shortest walk in $G$ that projects surjectively in both projections to $V(H)$ and its length correctly, starting from the given vertex $v$ in exponential time.
\end{theorem}

\begin{proof}
    
    The algorithm first looks for all the neighbours of $v$ in $G$. Then, the algorithm searches for the continuation of the walk. For each neighbour $a$ in the list that is not already visited, the walk is prolonged with $a$, and it is added to the visited list, so that the algorithm does not go in loops at first, and the edge is marked as visited edge list. If this new walk projects surjectively to $V(H)$ in both projections (this is done by a loop checking that every vertex of $H$ appears as first or second coordinate in the walk), then the length of the obtained walk is compared to what is currently the shortest walk projecting surjectively in both projections to $V(H)$ starting at $v$ and declared the new shortest walk if it is shorter. If the walk ending with $a$ does not project surjectively to all the vertices of $H$ in both projections, then starting from $a$, a new suitable shortest walk is searched for. Also, if a vertex is already visited, but there is a non visited edge to this vertex, the algorithm tries using it and tries to continue the walk by going in a cycle to other vertices. Finally, if this is also not possible, the algorithm tries going backwards along the walk until new unvisited vertex is found. As this process repeats for all the vertices in the neighbourhood, the total duration is upper bounded by $O(|V(G)|^{|V(G)|})$. This concludes the proof.
\end{proof}

\begin{algorithm}
\caption{minSteps(H,R)}
\label{alg:minSteps}
\KwIn{graph $H$, movement rules $R$}
\KwOut{The walk of shortest length projecting surjectively in both projections to $V(H)$ represented as a list}

\tcc{Depending on the rule set R, calculate the product}

\uIf{R is a traditional movement rule}{$G' = H \boxtimes H $}
\Else{
    \uIf{R is an active movement rule} {$G' = H \times H$}
    \Else {$G' = H \BoxProduct H$}
    }

\tcc{Compute the span of H, depending on the rule, and remove all vertices from G' whose distance is smaller than the span, i.e. construct G}

$sp = \verb|span(H, R)|$

$I = \emptyset$

\ForEach{(u,v) in $V(G')$}
{
    \If{$dist(u,v) \geq sp$}
    {add $(u,v)$ to $I$}
}

\tcc{Compute the graph induced by the vertices in I}

$G = G'\left[I\right]$

\tcc{Find the components that project surjectively to V(H) in both projections}
$ comp = \texttt{goodComponents(G)}$

  $minP = \emptyset$
  
  $visited = \emptyset$
  
\ForEach{c in comp}
{  

  \ForEach{node x in c}
  {
  \tcc{mp denotes the suitable shortest walk starting at x, ml denotes its length}
  [mp, ml] = \texttt{minLenWalk(G, H, x, minP, visited)}
  
  \If{mp projects surjectively to $V(H)$}
  {\If {$ml < |minP|$} {$minP = mp$}}
  }
}
\KwRet{minP}
\end{algorithm}

\begin{theorem}
    Algorithm~\ref{alg:minSteps} correctly finds a shortest walk projecting surjectively to $V(H)$ in both projections under the rules $R$ in exponential time. 
\end{theorem}
\begin{proof}
Let $H$ be a given graph with $n$ vertices. Algorithm~\ref{alg:minSteps} first computes the desired product depending on the rules (lines 1-7). Afterwards, using the function from~\cite{BaTa23}, the span of the graph $H$ is computed (line 8). Now, assuming we already have all pairs distances computed, we construct $G$ (lines 9-13). Then, only the components that project surjectively to $V(H)$ in both projections are considered for searching for a shortest walk that projects in both projections to $V(H)$ using Algorithm~\ref{alg:mnLenPath}, this implies the minimum number of steps required.

Computing the product graph and $G$ is done in time $O(n^4)$, as shown in~\cite{BaTa23}. Finding good components can be done also in polynomial time, as shown for Algorithm~\ref{alg:goodC}. However, this algorithm runs in exponential time, as the Algorithm~\ref{alg:mnLenPath} used for finding the minimum number of steps runs in exponential time. 
\end{proof}

\section*{Acknowledgements}
This work was initiated on the first Workshop on Games on Graphs on Rogla, Slovenia in 2023. The first and the third author acknowledge the financial support from the Slovenian Research Agency (research core funding No. P1-0297, also projects N1-0285 and BI-US/22-24-121). The second author acknowledges the partial financial support of the Provincial Secretariat for Higher Education and Scientific Research, Province of Vojvodina (Grant No.~142-451-2686/2021) and of Ministry of Science, Technological Development and Innovation of Republic of Serbia (Grants 451-03-66/2024-03/200125 \& 451-03-65/2024-03/200125).

\end{document}